\documentclass[12pt]{article}
\usepackage{fullpage}
\usepackage{amsfonts, amsmath, amssymb,latexsym,epsfig}
\usepackage{amsthm}
\usepackage{tikz}
\usetikzlibrary{graphs}
\usetikzlibrary{graphs.standard}
\usepackage{relsize}
\usepackage{verbatim}
\usepackage{enumerate}
\usepackage[skip=10pt plus1pt, indent=20pt]{parskip}

\usepackage{hyperref}
\usepackage{cleveref}

\newtheorem{thm}{Theorem}[section]

\newtheorem{lem}[thm]{Lemma}

\newtheorem{prop}[thm]{Proposition}
\newtheorem{conj}[thm]{Conjecture}

\newcommand{\F}{\mathbb F}

\author{
Vladislav Taranchuk \thanks{Department of Mathematics: Analysis, Logic and Discrete Mathematics, Ghent University, 9000 Ghent, Belgium. E-mail: {\tt vlad.taranchuk@ugent.be}.}
}

\title{$K_{2, t+1}$-free graphs containing an optimal number of $K_{t, t}$'s}

\begin{document}

\maketitle

\begin{abstract}
The generalized Tur\'{a}n number $ex(n, K_{t, t}, K_{2, t+1})$ is the maximum number of copies of $K_{t, t}$ that a $K_{2, t+1}$-free graph on $n$ vertices can contain. Recently, Pohoata, Tidor, and Yu established that $ex(n, K_{t, t}, K_{2, t+1}) = \Theta_t(n^2)$ for all integers $t \geq 3$. In this short note,  we use an explicit construction to establish that when $t$ is a prime power and $n = t^{2e  - 1}$ then
$$
ex(n, K_{t, t}, K_{2, t+1}) = (1 + o(1))\frac{n^2}{2t(t-1)}.
$$

\end{abstract}

\section{Introduction}

The generalized Tur\'{a}n number $ex(n, H, F)$ is the maximum number of copies of $H$ that an $F$-free graph on $n$ vertices can contain. This natural generalization of the standard Tur\'{a}n was first introduced by Alon and Shikhelman \cite{AlonShikhelman2016ManyTCopies}. There are many papers which are devoted to studying $ex(n, H, F)$. We refer the reader to the nice survey of Gerbner and Palmer \cite{GerbnerPalmer2025SurveyGeneralizedTuran} for a detailed history and up to date results on $ex(n, H, F)$. 

The more specific problem studying $ex(n, K_{a, b}, K_{c, d})$ has also seen considerable attention, however tight results are only known in particular regimes (see \cite[Section 4.3]{GerbnerPalmer2025SurveyGeneralizedTuran}). Recently, Pohoata, Tidor and Yu \cite{PohoataTidorYu2026K2tFreeGraphs} established that 
$$
ex(n, K_{t, t}, K_{2, t+1}) = \Theta_t(n^2),
$$
answering a question of Spiro. Their proof relies on a mix of geometric and probabilistic methods, constructing their graph from subspace-evasive sets in a 3-dimensional vector space over a finite field.

In this short note, we give a simple and completely explicit construction of $K_{2, t+1}$-free graphs for an infinite, albeit sparse, set of $n$'s, for which the corresponding graph contains an optimal number of $K_{t, t}$'s.

\begin{thm}\label{T: Main}
    Let $t$ be a prime power, $e > 1$ an integer and $n = t^{2e -1}$. Then 
    $$
    ex(n, K_{t, t}, K_{2, t+1}) = (1 + o(1))\frac{n^{2}}{2t(t-1)}
    $$
\end{thm}

The description of the graphs we use to obtain this result was first given by the author of this note with the application of the multicolor Ramsey number in mind for $K_{2, t+1}$-free graphs when $t$ is a prime power. This is mentioned in the survey of Lazebnik and Wang \cite{LazebnikWang2026TriangularSystemsSurvey}. We note however that a close inspection of the graphs constructed by Lazebnik and Mubayi \cite{LazebnikMubayi2002RamseyGraphsHypergraphs} in 2002 for the purpose of obtaining bounds on this same multicolor Ramsey number is in fact isomorphic to our graph.

\section{Proof of \Cref{T: Main}}
 We first establish a precise upper bound on $ex(n, K_{t, t}, K_{2, t+1})$. The proof of the following proposition follows that of \cite{PohoataTidorYu2026K2tFreeGraphs} where the authors observe that $ex(n, K_{t, t}, K_{2, t+1}) = O(n^2)$. A simple additional step beyond their argument gives the following upper bound.

\begin{prop}\label{P: UpBound}
    Let $t$ and $n$ be positive integers with $2\leq t \leq n$. Then
    $$
    ex(n, K_{t, t}, K_{2, t+1})\leq \frac{\binom{n}{2}}{2\binom{t}{2}}
    $$
\end{prop}

\begin{proof}
    Let $G$ be a $K_{2, t+1}$-free graph. Any two vertices $x$ and $y$ share a common neighborhood $N$ of at most $t$ elements, in which case the two vertices along with their common neighbors form a $K_{2, t}$ subgraph. We claim there is at most one $K_{t, t}$ which contains this copy of $K_{2, t}$ in $G$. Repeat this argument for any pair of vertices in $N$. This implies that there is at most one set $M$ with cardinality $t$ which contains $x$ and $y$ and for which the subgraph on $M$ and $N$ contains a $K_{t, t}$. 

    Counting over all pairs yields at most $\binom{n}{2}$ copies of $K_{t, t}$. However, this argument counts each $K_{t, t}$ precisely $2\binom{t}{2}$ times which is the number of ways to pick two vertices from the same side of a given $K_{t, t}$. Hence, 
    $$
    ex(n, K_{t, t}, K_{2, t+1})\leq \frac{\binom{n}{2}}{2\binom{t}{2}}.
    $$
\end{proof}

We now give an explicit construction which attains this upper bound (modulo lower order terms) when $t$ is any prime power and $n$ is any odd power of $t$. 

\noindent \textbf{The graph:} Let $t$ be a prime power, $q = t^e$ for any integer $e \geq 1$. Let $L$ be an $\F_t$-linearized polynomial in $\F_q[x]$ which has $\dim (\ker (L)) = 1$ and denote by $R_L$ the image space of $L$. We note that $L$ can be thought of as a linear transformation of rank $e - 1$ on the vector space $\F_t^e$. 
Since $L$ is $\F_t$-linear, we know that $R_L$ is closed under addition and that  $|R_L| = t^{e-1}$. 
Define the graph $\Gamma_{t, e}$ with vertex set $V = \F_{q}\times R_{L}$. Two distinct vertices $(v_1, v_2), (w_1, w_2) \in V$ are adjacent if and only if 
\begin{align}\label{adjeq}
v_2 + w_2 = L(v_1w_1).
\end{align}
Note that $|V| = q|R_L| = q^2/t$.

\noindent \textbf{Remark}: Although we work with simple graphs, the adjacency equation implies that a vertex can sometimes be adjacent to itself. We remove any such loops. We will say that any two (not necessarily distinct) vertices $(v_1, v_2)$ and $(w_1, w_2)$ are \textit{algebraically adjacent} if they satisfy (\ref{adjeq}).

\begin{lem}
    $\Gamma_{t, e}$ is $K_{2, t+1}$-free.
\end{lem}

\begin{proof}
    Observe that given any vertex $v = (v_1, v_2)$, and any $\alpha \in \F_{q}$, $\alpha \neq v_1$, the vertex $(v_1, v_2)$ has a unique neighbor with first coordinate $\alpha$. In particular, using (\ref{adjeq}), we have that $v$ is adjacent to $(\alpha, L(\alpha v_1) - v_2)$. 

    Let $v = (v_1, v_2)$ and $w = (w_1, w_2)$ be two vertices in $\Gamma$, we will show that they cannot have more than $t$ common neighbors. If $v_1 = w_1$, then $v$ and $w$ share no common neighbors because if they did, it would imply that this common neighbor is adjacent to two distinct vertices with same first coordinate, a contradiction. So we suppose that $v_1 \neq w_1$.

    Let $x = (x_1, x_2) \in V$ be a common neighbor of both $v$ and $w$. Then the equations
    \begin{align}
        x_2 + v_2 = L(x_1v_1) \\
        x_2 + w_2 = L(x_1w_1)
    \end{align}
    must be satisfied by $(x_1, x_2)$. Subtracting equation (3) from equation (2), we obtain that 
    \begin{align}\label{eq2}
        v_2 - w_2 = L(x_1(v_1 - w_1)).
    \end{align}
    Thus if $x$ is adjacent to $v$ and $w$, its first coordinate must be a solution to (\ref{eq2}). Since $v_1 \neq w_1$, then we know that the polynomial on the right hand side is not 0. Let $a = v_1 - w_1$ and $b = v_2 - w_2$, then $x_1$ satisfies
    $$
    L(a x_1) = b
    $$
    Since $b \in R_L$, then we know that there exists at least one solution. On the other hand, $L$ is $\F_t$-linear over $\F_{q}$ and has $t$ roots and so must in fact be a $t$-to-$1$ mapping. Therefore, there are exactly $t$ different values of $x_1$ which satisfy (\ref{eq2}), and so $v$ and $w$ have at most $t$ common neighbors. Thus $\Gamma_{t, e}$ is $K_{2, t+1}$-free.
\end{proof}

We now establish the following lemma which is the key ingredient for the proof of the main theorem.

\begin{lem}\label{L: Ktt}
    Let $(v_1, v_2), (w_1, w_2)$ be two non-adjacent vertices in $\Gamma_{t, e}$ with $v_1 \neq w_1$. Then there exist two sets of vertices $M$ and $N$ of cardinality $t$, with $(v_1, v_2), (w_1, w_2)\in M$, such that each vertex in $M$ is algebraically adjacent to each vertex in $N$.
\end{lem}

\begin{proof}
    Recall that the set of common neighbors $(x_1, x_2)$ of $(v_1, v_2)$ and $(w_1, w_2)$ is given by 
    $$
    N = \{ (x_1, L(w_1x_1) - w_2):v_2 - w_2 = L(x_1(v_1 - w_1)) \}
    $$
    which has exactly $t$ elements. We also highlight that the set $N$ forms an $\F_t$-line in $\F_{q} \times R_L$ because the set of solutions $x_1$ to $L(x_1(v_1 - w_1) = v_2 - w_2$ forms an $\F_t$-line in $\F_q$.
    
    Since $v_1 \neq w_1$, then we know that $N$ is not empty. Furthermore, since $(v_1, v_2)$ is not adjacent to $(w_1, w_2)$, then $(v_1, v_2), (w_1, w_2) \not\in N$. Hence, $N$, together with the vertices $(v_1, v_2), (w_1, w_2)$ forms a $K_{2, t}$ in $\Gamma_{t, e}$. Now, let $M$ be the $\F_t$-line in $\F_q \times R_L$ which contains the points $(v_1, v_2), (w_1, w_2)$. In particular, 
    $$
    M = \{ (w_1 + m(v_1 - w_1), w_2  + m(v_2 - w_2)): m \in \F_t \}.
    $$
    We claim that all vertices in $M$ are algebraically adjacent to all vertices in $N$. We have $(w_1 + m(v_1 - w_1), w_2  + m(v_2 - w_2)) \in M$ is adjacent to $(x_1, L(w_1x_1) - w_2) \in N$ if 
    $$
    w_2 + m(v_2 - w_2) + L(w_1x_1) - w_2 = L(x_1(w_1 + m(v_1 - w_1))).
    $$
    Now 
    $$
    m(v_2 - w_2)  = mL(x_1(v_1 - w_1)) = L(mx_1(v_1 - w_1)).
    $$
    Hence
    $$
    w_2 + m(v_2 - w_2) + L(w_1x_1) - w_2  = L(mx_1(v_1 - w_1)) + L(w_1x_1) = L(x_1(w_1 + m(v_1 - w_1)))
    $$
    Therefore every vertex of $M$ is algebraically adjacent to every vertex of $N$.

    Since $\Gamma_{t, e}$ is $K_{2, t+1}$-free, the set $N$ is uniquely determined. Likewise, choosing any two distinct vertices in $N$ uniquely determines $M$ for the same reason.
\end{proof}

In the lemma above, we assumed that the two given vertices are not adjacent. Since $M$ and $N$ are $\F_t$-lines, it follows that $M \neq N$ and hence we  either have that $|M \cap N| = 1$ or $|M \cap N| = 0$.  When $|M\cap N| = 0$, the subgraph with parts $M$ and $N$ contains a $K_{t, t}$. On the other hand when $|M \cap N| = 1$, the subgraph with parts $M$ and $N$ contains the tripartite complete graph $K_{1, t-1, t-1}$. 

\begin{proof}[Proof of \Cref{T: Main}]
\Cref{L: Ktt} implies that any two distinct non-adjacent vertices $(v_1, v_2)$ and $(w_1, w_2)$ with $v_1 \neq w_1$ uniquely define either a $K_{t, t}$ or a $K_{1, t-1, t-1}$. Indeed, once a valid pair of vertices is chosen and fixed, the corresponding sets $M$ and $N$ are obtained deterministically. The number of such distinct pairs of vertices is at least $t^{2e - 1}(t^{2e - 1} - t^e - t^{e-1})/2$ since the degree of each vertex is at most $t^e$ and we also exclude all vertices whose first coordinate is $v_1$. Each distinct $K_{t, t}$ and $K_{1, t-1, t-1}$ is counted at most $2\binom{t}{2}$ times. As $n = t^{2e - 1}$, the total number of distinct $K_{t, t}$'s and $K_{1, t-1, t-1}$'s in $\Gamma_{t, e}$ which are included in the count is at least 
$$
\frac{t^{2e - 1}(t^{2e - 1} - t^e - t^{e-1})}{2t(t-1)} = (1 + o(1))\frac{n^2}{2t(t-1)}
$$
We now show that the number of distinct $K_{1, t-1, t-1}$'s which are counted in this way is negligible compared to the total count.

Observe that if $|M \cap N| = 1$, then the intersection $M\cap N$ is a vertex $v = (v_1, v_2)$ which is algebraically adjacent to itself, and therefore satisfies $2v_2 = L(v_1^2)$. This implies that the symmetric difference $M \Delta N$ is entirely contained in the neighborhood of $v$. Hence, to upper bound the total number of distinct $K_{1, t-1, t-1}$ which appear in our count, we may simply count pairs $\{v, w\}$ where $v$ is algebraically adjacent to itself and $w$ is a neighbor of $v$. Indeed, for any such bad configuration and any $w\in M\triangle N$, the $\F_t$-line containing $v$ and $w$
is fixed, while the other $\F_t$-line is the set of algebraic common neighbors of $v$
and $w$, which is uniquely determined by the same common-neighborhood
calculation as in \Cref{L: Ktt}.

The number of vertices $v = (v_1, v_2)$ which satisfy $2v_2 = L(v_1^2)$ is $|\F_q| = t^e$. Thus the total number of such pairs is at most $t^e\cdot t^e = tn$. Therefore the number of $K_{t, t}$'s is at least $(1+ o(1))\frac{n^2}{2t(t-1)}$, finishing the proof.
\end{proof}

\section{Concluding Remarks}

The exact lower bound obtained in \cite{PohoataTidorYu2026K2tFreeGraphs} is of the form 
$$
(1 - o(1))\frac{n^2}{4t^2t!} \leq ex(n, K_{t, t}, K_{2, t+1}) .
$$
Our results in this paper, together with the monotonicity of $ex(n, H, F)$, imply that for all prime powers $t$, we have 
$$
 (1 - o(1))\frac{n^2}{2t^5(t-1)} \leq ex(n, K_{t, t}, K_{2, t+1}). 
$$
This improves upon the former bound when $t > 5$ is a prime power.

We end by conjecturing that the upper bound from \Cref{P: UpBound} gives the true growth of $ex(n, K_{t, t}, K_{2, t+1})$ modulo the lower order terms.

\begin{conj}
    $$
    ex(n, K_{t, t}, K_{2, t+1}) = (1 + o(1))\frac{n^2}{2t(t-1)}.
    $$
\end{conj}

\bibliographystyle{plainurl}
\bibliography{references.bib}

\end{document}